\documentclass[preprint,12pt]{elsarticle}

\usepackage{amsmath}
\usepackage{amsfonts}
\usepackage{amssymb}
\usepackage{graphicx}
\usepackage{amsthm}

\journal{}
\newtheorem{theorem}{Theorem}[section]

\newtheorem{definition}[theorem]{Definition}

\newtheorem{remark}[theorem]{Remark}

\numberwithin{equation}{section}

\begin{document}

\begin{frontmatter}


\title{On the best Ulam constant of a higher order linear difference equation}


\author{Alina Ramona Baias and Dorian Popa}

\address{Technical University of Cluj-Napoca,
Department of Mathematics \\
G. Bari\c tiu No.25, 400027, Cluj-Napoca, Romania e-mail:
\textrm{Baias.Alina@math.utcluj.ro.}}

\address{Technical University of Cluj-Napoca,
Department of Mathematics \\
G. Bari\c tiu No.25, 400027, Cluj-Napoca, Romania e-mail:
\textrm{Popa.Dorian@math.utcluj.ro}}

\begin{abstract}
In a Banach space $X$ the linear difference equation with constant coefficients $x_{n+p}=a_1x_{n+p-1}+\ldots+a_px_n,$ is Ulam stable if and only if the roots $r_k,$ $1\leq k\leq p,$ of its characteristic equation do not belong to the unit circle. If $|r_k|>1,$ $1\leq k\leq p,$ we prove that the best Ulam constant of this equation is $\frac{1}{|V|}\sum\limits_{s=1}^{\infty}\left|\frac{V_1}{r_1^s}-\frac{V_2}{r_2^s}+\ldots+\frac{(-1)^{p+1}V_p}{r_p^s}\right|,$ where $V=V(r_1,r_2,\ldots,r_p)$ and $V_k=V(r_1,\ldots,r_{k-1},r_{k+1}, \ldots, r_p),$ $1\leq k\leq p,$ are Vadermonde determinants.
\end{abstract}


\end{frontmatter}



\section{Introduction}

The origin of the stability theory of functional equations is traced to the fall of 1940, when S.M. Ulam considered the problem of approximate homomorphisms of groups. The first partial answer to Ulam$'$s problem came within a year and it was given by D.H. Hyers, who proved that
Cauchy's equation in Banach spaces is stable \cite{hyers}.

Generally, we say that an equation is stable in Ulam sense if for every approximate solution of it
there exists an exact solution of the equation near it. For more details and
results on Ulam stability we refer the reader to \cite{agarwal, brzdek-popa, hyers1}.

In the last years, results on Ulam stability have been obtained in various directions, we mention here on one hand the stability results for functional, difference or differential equations \cite{brzdeck_, baias-popa, baias-popa1, baias-popa2, popa1} and on the other hand the Ulam stability results for linear operators \cite{brzdek-popa, popa12}.

The problem of the best Ulam constant was first posed in \cite{Rassias}. In the
literature there are only a few results on the best Ulam constant of
equations and operators. We can merely mention here the characterization of Ulam stability of
linear operators and the representation of their best Ulam constants obtained in \cite{popa3, popa4, popa12}. In the same direction J. Brzdek, S.M. Jung and M. Th Rassias gave sharp estimates for the Ulam constant of some second order linear difference equations \cite{brzdek-jung, brzdek12, jung, rassias1}. S.J Dilworth et all. in \cite{dilworth} obtained the best Ulam constant of approximately convex functions. Later, M. Onitsuka \cite{onitsuka1, onitsuka} and D.R. Anderson and M. Onitsuka \cite{anderson} obtained results on Hyers-Ulam stability and on the best Ulam constant for a first order and a second order linear difference equations with constant stepsize. C. Bu\c se et all. \cite{buse2} proved that a discrete system $X_{n+1}=AX_n,$ $n\in\mathbb{N},$ where $A$ is a matrix with complex entries, is Ulam stable if and only if $A$ possesses a discrete dichotomy. Recently, A.R. Baias and D. Popa obtained the best Ulam constant for a second and a third order linear difference equation in Banach spaces \cite{baias-popa, baias-popa1}, as well as for the second order linear differential operator \cite{baias-popa2}.

The discrete dynamical systems are governed by difference equations. The notions of stability and asymptotic stability for such systems concern the behaviour of the solutions of the associated difference equation with respect to an equilibrium point (see \cite{elaydi}). Ulam stability is connected to the notion of perturbation and shadowing of a discrete dynamical system, so it can be a measure of the reaction of the system under perturbation (see \cite{palmer}). Finding the best Ulam constant in this case it means to find the closest exact solution of the dynamical system to a solution of the perturbed system.

The goal of this paper is to determine the best Ulam constant for a $p$ order linear difference
equation with constant coefficients in Banach spaces, for distinct roots of the characteristic equation belonging to the exterior of the unit disc. In this way we improve and complement some existent results in the field.

\section{Main results}

Let $\mathbb{K}$ be either the field of real numbers $\mathbb{R}$ or the field of complex numbers $\mathbb{C}$ and $(X,\|\cdot\|)$ a Banach space over the field $\mathbb{K}.$ By $\mathbb{N}=\{0,1,2,\ldots\}$  we denote the set of all nonnegative integers.

Consider the linear difference equation of order $p$
\begin{equation}\label{eq1}
x_{n+p}=a_1x_{n+p-1}+\ldots+a_px_n,\ n\in\mathbb{N},
\end{equation}
where $a_1,a_2,\ldots, a_p\in\mathbb{K},$ $x_0,\ldots,x_{p-1}\in X$ and $p\geq 1$ is a positive integer.

We denote by $r_1,r_2,\ldots,r_p$ the complex roots of the characteristic equation associated to $(\ref{eq1})$, i.e.,
\begin{equation}\label{eqchar}
r^p=a_1r^{p-1}+\ldots+a_p.
\end{equation}
If $r_1,r_2,\ldots,r_p $ are distinct numbers, then the general solution of the equation $(\ref{eq1})$ is given by

 \begin{equation}\label{eq2-solgen}
 x_n^{(H)}=\mathcal{C}_1r_1^n+\ldots+\mathcal{C}_p r_p^n,\quad n\in\mathbb{N},
 \end{equation}
 where $\mathcal{C}_1,\ldots,\mathcal{C}_p\in X$ are arbitrary constants (see \cite{cull, elaydi, gil}). For more details and results on linear difference equations in Banach spaces see \cite[Chapter 6]{gil}.

\begin{definition}\label{def-sability}
The equation $(\ref{eq1})$ is called \emph{Ulam stable} if there exists a constant $K\geq0$ such that for every $\varepsilon>0$ and every sequence $(x_n)_{n\geq 0}$ in $X$ satisfying
\begin{equation}\label{eq2}
\|x_{n+p}-a_1x_{n+p-1}-\ldots-a_px_n\|\leq\varepsilon,\  n\in\mathbb{N},
\end{equation}
there exists a sequence $(y_n)_{n\geq 0}$ in $X$ such that
\begin{equation}\label{eq3}
y_{n+p}=a_1y_{n+p-1}+\ldots+a_py_n,\ n\in\mathbb{N},
\end{equation}
\begin{equation}\label{eq4}
\|x_n-y_n\|\leq K\varepsilon,\ n\in\mathbb{N}.
\end{equation}
\end{definition}

A sequence $(x_n)_{n\geq 0}$ satisfying $(\ref{eq2})$ for some positive $\varepsilon$ is called an \emph{approximate solution} of the equation $(\ref{eq1nonh}).$ So, Definition \ref{def-sability} can be reformulated as follows: the equation $(\ref{eq1})$ is called Ulam stable if for every approximate solution of $(\ref{eq1})$ there exists an exact solution near it.
The number $K$ from Definition \ref{def-sability} is called an \emph{Ulam constant} of the equation $(\ref{eq1}).$
In what follows we will denote by $K_R$ the infimum of all Ulam constants of the equation $(\ref{eq1}).$ If $K_R$ is an Ulam constant for $(\ref{eq1})$ then we call it \emph{the best Ulam constant} or the \emph{Ulam constant} of the equation.
Generally, the infimum of all Ulam constants of an equation is not necessary an Ulam constant of that equation (see \cite{hatori, popa1}).

Some Ulam stability results for the equation $(\ref{eq1})$ can be found in \cite{brzdek-popa}. Here we recall a result obtained in \cite{brzdek-popa1} and \cite{popa}.
\begin{theorem}
\begin{enumerate}
\item[$i)$] If $|r_k|\neq 1$ for every $k=1,2,\ldots,p,$ then for every $\varepsilon>0$ and every sequence $(x_n)_{n\geq0}$ in $X$ satisfying
\begin{equation}
\|x_{n+p}-a_1x_{n+p-1}-\ldots-a_px_n\|\leq\varepsilon,\ n\in\mathbb{N},
\end{equation}
there exists a sequence $(y_n)_{n\geq0}$ in $X$ such that
\begin{eqnarray*}
y_{n+p}=a_1y_{n+p-1}+\ldots+a_py_n,\ n\in\mathbb{N},\\
\|x_n-y_n\|\leq\frac{\varepsilon}{|\prod\limits_{k=1}^{p}(|r_k|-1)|},\ n\in\mathbb{N}.
\end{eqnarray*}
Moreover, if $|r_k|>1$ for all $k=1,2,\ldots,p,$ then the sequence $(y_n)_{n\geq0}$ is unique.
\item[$ii)$] If there exists $j\in\{1,\ldots,p\}$ such that $|r_j|=1,$ then the equation $(\ref{eq1})$ is not Ulam stable.
\end{enumerate}
\end{theorem}
 Remark that $K=\frac{1}{\left|\prod\limits_{k=1}^{p}(|r_k|-1)\right|}$ is an Ulam constant for the linear difference equation $(\ref{eq1})$.

Consider also the linear and nonhomogeneous equation

\begin{equation}\label{eq1nonh}
x_{n+p}=a_1x_{n+p-1}+\ldots+a_px_n+f_n,\quad n\in\mathbb{N},
\end{equation}
associated to $(\ref{eq1}),$ where $(f_n)_{n\geq 0}$ is a sequence in $X.$
The general solution of $(\ref{eq1nonh})$ is given by
\begin{equation}\label{eq gen}
x_n=x_n^{(H)}+x_n^{(P)},\quad n\in\mathbb{N},
\end{equation}
where $x_n^{(H)}$ is the general solution of the homogeneous equation $(\ref{eq1})$ and $x_n^{(P)}$ is a particular solution of $(\ref{eq1nonh}).$
According to the method of variation of parameters (see \cite{elaydi, gil}), the equation $(\ref{eq1nonh})$  admits a particular solution of the form
 \begin{equation}\label{eq3-solpart}
 x_{n}^{(P)}=\mathcal{C}_1(n) r_1^n+\ldots+\mathcal{C}_p(n)r_p^n, \quad n\in\mathbb{N}.
 \end{equation}
 The coefficients $\mathcal{C}_1(n),\ldots, \mathcal{C}_p(n)$ satisfy the equation
 \begin{equation}\label{eq3-matrice}
 \begin{pmatrix}r_1^{n+1}&r_2^{n+1}&\ldots& r_p^{n+1}\\r_1^{n+2}&r_2^{n+2}&\ldots& r_p^{n+2}\\ \ldots &\ldots& \ldots& \ldots\\
 r_1^{n+p}&r_2^{n+p}&\ldots& r_p^{n+p}\end{pmatrix}
 \begin{pmatrix}
   \Delta\mathcal{C}_{1}(n) \\
    \Delta\mathcal{C}_{2}(n) \\
    \vdots\\
    \Delta\mathcal{C}_{p}(n)
 \end{pmatrix}=
 \begin{pmatrix}
   0 \\
   \vdots \\
   0 \\
   f_n
 \end{pmatrix}, \quad n\in\mathbb{N},
 \end{equation}
 where $\Delta f$ denotes the linear difference of order one, i.e., $\Delta f(x)=f(x+1)-f(x),$ for a function $f:\mathbb{R}\to X.$
Denoting
\begin{eqnarray*}
W(n+1)= \begin{pmatrix}r_1^{n+1}&r_2^{n+1}&\ldots& r_p^{n+1}\\r_1^{n+2}&r_2^{n+2}&\ldots& r_p^{n+2}\\ \ldots &\ldots& \ldots& \ldots\\
 r_1^{n+p}&r_2^{n+p}&\ldots& r_p^{n+p}\end{pmatrix},
 \end{eqnarray*}
\begin{eqnarray*}
  X(n)=\begin{pmatrix}
   \Delta\mathcal{C}_{1}(n) \\
    \Delta\mathcal{C}_{2}(n) \\
    \vdots\\
    \Delta\mathcal{C}_{p}(n)
 \end{pmatrix}\mbox{ and }
 F(n)=\begin{pmatrix}
   0 \\
   \vdots \\
   0 \\
   f_n
 \end{pmatrix},
 \end{eqnarray*}
the equations $(\ref{eq3-matrice}) $ becomes
$$W(n+1)\cdot X(n)=F(n),\quad n\in\mathbb{N}.$$
Consequently,
$$ X(n)=W^{-1}(n+1)\cdot F(n), \quad n\in\mathbb{N}.$$
The inverse matrix is given by $W^{-1}(n+1)=\frac{1}{\det W(n+1)}W^{*}(n+1),$ where
$W^{*}(n+1)=(w_{i,j}),$ $i,j=1,\ldots,p$ denotes the adjoint matrix.
Notice that in order to find the value of the coefficients $\mathcal{C}_1(n),\ldots,\mathcal{C}_p(n)\in X,$ one needs to find only the elements $(w_{i,p}),$ $i=1,\ldots,p$ of the adjoint matrix $W^{*}(n+1)$, i.e., the cofactors $(p,j)_{1\leq j\leq p}$ of the matrix $W(n+1).$

In what follows, we denote for simplicity the Vandermonde determinants of order $p+1$ by $V=V(r_1,r_2,\ldots,r_p)$ and by $V_k=V(r_1,r_2,\ldots,r_{k-1},r_{k+1},\ldots, r_p),$ $k=1,\ldots, p,$ the Vandermonde determinants of order $p,$ respectively. Consequently, we obtain $$det W(n+1)=r_1^{n+1}r_2^{n+1}\cdots r_{p}^{n+1}\cdot V(r_1,r_2,\ldots,r_p)$$ and
$$C_{k}(n)=(-1)^{p+k}\frac{V_k}{V}\sum\limits_{s=1}^{n}\frac{f_{s-1}}{r_k^{s}},\ k=1,\ldots,p.$$

Hence, a particular solution of the equation $(\ref{eq1nonh})$ takes the following
form
\begin{equation}\label{forma sol particulara}
x_n^{(P)}=\frac{1}{V}\sum\limits_{s=1}^{n}\{(-1)^{p+1}V_1r_1^{n-s}+(-1)^{p+2}V_2r_2^{n-s}+\ldots+(-1)^{2p}V_pr_p^{n-s}\}f_{s-1}.
\end{equation}

The main result on the Ulam stability of the equation $(\ref{eq1})$ is given in the next theorem.

\begin{theorem}\label{th1}
Suppose that the characteristic equation admits distinct roots with $|r_k|> 1,$ $k=1,\dots,p.$
Then for every $\varepsilon>0$ and every sequence $(x_n)_{n\geq0}$ in $X$ satisfying
\begin{equation}\label{eq6}
\|x_{n+p}-a_1x_{n+p-1}-\ldots-a_px_n\|\leq\varepsilon,\  n\in\mathbb{N},
\end{equation}
there exists a unique sequence $(y_n)_{n\geq 0}$ in $X$ such that
\begin{equation}\label{eq7}
y_{n+p}=a_1y_{n+p-1}+\ldots+a_py_n,\ n\in\mathbb{N},
\end{equation}
\begin{equation}\label{eq8}
\|x_n-y_n\|\leq K\varepsilon,\ n\in\mathbb{N},
\end{equation}
where
\begin{equation}\label{constK}
K=\frac{1}{|V|}\sum\limits_{s=1}^{\infty}\left|\frac{V_1}{r_1^s}-\frac{V_2}{r_2^s}+\ldots+(-1)^{p+1}\frac{V_p}{r_p^s}\right|.
\end{equation}
\end{theorem}
\begin{proof}{\bf{Existence.}} First we consider the case $\mathbb{K}=\mathbb{C}.$
Let $(x_n)_{n\geq0}$ be a sequence in $X$ satisfying $(\ref{eq6})$ and let
$$f_n:=x_{n+p}-a_1x_{n+p-1}-\ldots-a_px_n,\quad n\in\mathbb{N}.$$
Hence $\|f_n\|\leq \varepsilon,$ for every $n\in\mathbb{N}.$

Then, there exist $\mathcal{C}_1,\ldots,\mathcal{C}_p \in X$ such that
$$x_n= \mathcal{C}_1r_1^n+\ldots+\mathcal{C}_pr_p^n+x_n^{(P)},$$
where $x_n^{(P)}$ is given by $(\ref{forma sol particulara}).$
Define $(y_n)_{n\geq 0}$  by $y_0=x_0$ and $$y_n=\overline{\mathcal{C}}_1r_1^n+\ldots+\overline{\mathcal{C}}_pr_p^n, \ n\geq 1,$$
with
\begin{equation}\label{alegere ck}
\overline{\mathcal{C}}_k=\mathcal{C}_k+\frac{(-1)^{p+k}V_k}{V}\sum\limits_{s=1}^{\infty}\frac{f_{s-1}}{r_k^{s}},\ 1\leq k\leq p.
\end{equation}
Since $\|\frac{f_{s-1}}{r_k^{s}}\|\leq \frac{\varepsilon}{|r_k|^{s}},$ $s\geq 1$ and $|r_k|>1,$ $k=1,\ldots,p,$ it follows that the series $\sum\limits_{s=1}^{\infty}\frac{f_{s-1}}{r_k^{s}}$ is absolutely convergent, so the constants $\overline{\mathcal{C}}_k,$ $k=1,\ldots, p,$ are well defined. On the other hand $(y_n)_{n\geq 0}$ satisfies (\ref{eq7}).

We get
\begin{eqnarray*}
x_{n}-y_{n}\!&=&\!\!-\frac{1}{V}\left\{\sum\limits_{s=n+1}^{\infty}\left((-1)^{p+1}\frac{V_1}{r_1^{s-n}}+(-1)^{p+2}\frac{V_2}{r_2^{s-n}}+\ldots+(-1)^{2p}\frac{V_p}{r_p^{s-n}}\right)f_{s-1}\right\}\\
&=&\!\!-\frac{1}{V}\sum\limits_{s=1}^{\infty}\left((-1)^{p+1}\frac{V_1}{r_1^{s}}+(-1)^{p+2}\frac{V_2}{r_2^{s}}+\ldots(-1)^{2p}\frac{V_p}{r_p^{s}}\right)f_{s+n-1},\ n\in\mathbb{N},
\end{eqnarray*}
therefore
\begin{eqnarray*}
\|x_n-y_n\|\leq \frac{\varepsilon}{|V|}\sum\limits_{s=1}^{\infty}\left|\frac{V_1}{r_1^s}-\frac{V_2}{r_2^s}+\ldots+(-1)^{p+1}\frac{V_p}{r_p^s}\right|,\ n\in\mathbb{N}.
\end{eqnarray*}

Now let $\mathbb{K}=\mathbb{R}.$ Then $X^2$ is a complex Banach space with a linear structure and the Taylor norm $\|\cdot\|_{T}$ defined by
\begin{align*}
(x,y)+(z,w)&=(x+z,y+w)\\
(\lambda+i\mu)(x,y)&=(\lambda x-\mu y, \mu x+\lambda y)\\
\|(x,y)\|_{T}&=\sup\limits_{0\leq\theta\leq 2\pi}\|(\cos \theta)x+(\sin \theta) y\|,
\end{align*}
for $x,y,z,w\in X$ and $\lambda, \mu\in\mathbb{R},$ see \cite[p. 66]{kadison} and \cite[p. 39]{fabian}.
The following relations hold
$$\max\{\|x\|,\|y\|\}\leq \|(x,y)\|_{T}\leq \|x\|+\|y\|$$ for all $x,y\in X.$ Define $(X_{n})_{n\geq 0}$ by $(X_{n})=(x_{n},0),$ $n\in\mathbb{N}.$ Then
$$\|X_{n+p}-a_1X_{n+p-1}-\ldots-a_pX_{n}\|_{T}\leq\varepsilon,\ n\in\mathbb{N}.$$
According to the previous part of the proof there exists a sequence $(Y_n)_{n\geq0}$ in $X^2$ such that
$$\|X_n-Y_n\|_{T}\leq K\varepsilon,\ n\in\mathbb{N}.$$

Let $p_i(x_1,x_2)=x_i,$ $i=1,2.$ Then $y_n=p_1(Y_n),$ $n\in\mathbb{N},$ is a solution of $(\ref{eq1})$ and $(\ref{eq8})$ holds.

{\bf Uniqueness.} Suppose that for a sequence $(x_n)_{n\geq0}$ satisfying $(\ref{eq6})$ there exist two sequences $(y_n)_{n\geq0}$ and $(z_n)_{n\geq 0}$ satisfying the equation $(\ref{eq7})$ such that
$$\|x_n-y_n\|\leq K\varepsilon \mbox{ and }\|x_n-z_n\|\leq K\varepsilon,\ n\geq0,$$
where $K$ is given by $(\ref{constK}).$ Then
\begin{eqnarray}\label{marginire}
\|y_n-z_n\|\leq\|y_n-x_n\|+\|x_n-z_n\|\leq2K\varepsilon,\ n\geq0.
\end{eqnarray}
The sequence $(u_n)_{n\geq0},$  $u_n=y_n-z_n,$ $n\geq0,$ satisfies also the relation $(\ref{eq1}),$ therefore there exist $\lambda_1,\lambda_2,\ldots,\lambda_p\in X$ such that
\begin{eqnarray*}
u_n=\lambda_1r_1^n+\lambda_2r_2^n+\ldots+\lambda_pr_p^n.\\
\end{eqnarray*}

Since $|r_k|>1,$ $k=1,\ldots,p,$ it follows that $(u_n)_{n\geq0}$ is unbounded,a contradiction to $(\ref{marginire}).$
Consequently, $\lambda_1=\lambda_2=\dots=\lambda_p=0,$ which entails $y_n=z_n,$ $n\geq0.$
The theorem is proved.
\end{proof}

The result on the \emph{best Ulam constant} of the equation $(\ref{eq1})$ is given in the next theorem.

\begin{theorem}\label{th2}
If $|r_k|>1,$ $1\leq k\leq p,$ then the best Ulam constant of the equation $(\ref{eq1})$ is given by
\begin{equation}\label{bestct}
K_R=\frac{1}{|V|}\sum\limits_{s=1}^{\infty}\left|\frac{V_1}{r_1^s}-\frac{V_2}{r_2^s}+\ldots+\frac{(-1)^{p+1}V_p}{r_p^s}\right|.
\end{equation}
\end{theorem}
\begin{proof}
Suppose that the equation $(\ref{eq1})$ admits an Ulam constant $K<K_R.$ Let $\varepsilon>0,$ $u\in X,$ $\|u\|=1 $ and \begin{eqnarray}\label{def-fn}
f_n=
\left\{\begin{array}{ll}
\frac{|E_n|}{E_n}u\varepsilon,& \mbox{ if } E_n\neq 0,\\
0,& \mbox{ if } E_n=0,
\end{array} \right.
\end{eqnarray}
where $$E_n=\frac{V_1}{r_1^n}-\frac{V_2}{r_2^n}+\ldots+\frac{(-1)^{p+1}V_p}{r_p^n},\ n\geq 1.$$

Let $(x_n)$ be the solution of the equation
$$x_{n+p}-a_1x_{n+p-1}-\ldots -a_p x_n=f_n,\ n\geq 0,$$
given by
$$x_n=\mathcal{C}_1r_1^{n}+\ldots+\mathcal{C}_pr_p^{n}+x_n^{(P)},$$ where $x_n^{(P)}$ is given by $(\ref{eq3-solpart})$ with
$$C_k=-(-1)^{p+k}\frac{V_k}{V}\sum\limits_{s=1}^{\infty}\frac{f_{s-1}}{r_k^s},\quad k=1,2,\ldots, p.$$
Then
\begin{eqnarray}
x_n&=&-\frac{1}{V}\sum\limits_{s=n+1}^{\infty}\left(\frac{(-1)^{p+1}V_1}{r_1^{s-n}}+\frac{(-1)^{p+2}V_2}{r_2^{s-n}}+\cdots +\frac{(-1)^{2p}V_p}{r_p^{s-n}}\right)f_{s-1}\nonumber \\
&=& \frac{(-1)^p}{V}\sum\limits_{s=1}^{\infty}\left(\frac{V_1}{r_1^{s}}-\frac{V_2}{r_2^{s}}+\cdots +\frac{(-1)^{p+1}V_p}{r_p^{s}}\right)f_{n+s-1}.\label{eq11}
\end{eqnarray}

Since $\|f_n\|\leq \varepsilon,$ $n\geq 0,$ and $|r_k|>1,$ $1\leq k\leq p,$ it follows that $(x_n)_{n\geq 0}$ is a bounded sequence in $X$ and
$$\|x_{n+p}-a_1 x_{n+p-1}-\cdots-a_{p}x_n\|\leq \varepsilon,\ n\geq 0.$$

Then there exist a sequence $(y_n)_{n\geq 0}$ satisfying $(\ref{eq3}),$  $y_n=\mathcal{K}_1 r_1^n+\ldots+ \mathcal{K}_p r_p^{n},$ $n\geq 0,$ $\mathcal{K}_1,\ldots,\mathcal{K}_p\in X,$ such that
\begin{equation}\label{eq 12}
\|x_n-y_n\|\leq K\varepsilon,\ n\geq 0.
\end{equation}

If $(\mathcal{K}_1, \mathcal{K}_2,\ldots, \mathcal{K}_n)\neq (0,0,\ldots,0)$ letting $n\to\infty$ in $(\ref{eq 12})$ we get $\infty\leq K\varepsilon,$ a contradiction. Therefore
$(\mathcal{K}_1, \mathcal{K}_2,\ldots, \mathcal{K}_n)=(0,0,\ldots,0),$ $y_n=0,$ for all $n\geq 0.$ For $n=1$ in $(\ref{eq 12})$ it follows that $\|x_1\|\leq K\varepsilon.$ But according to $(\ref{eq11}),$ we get
$$x_1=\frac{(-1)^p}{V}\sum\limits_{s=1}^{\infty}E_sf_s=\frac{(-1)^p}{V}u\varepsilon\sum\limits_{s=1}^{\infty}|E_s|.$$
Therefore $\|x_1\|=\varepsilon K_R.$ Thus the relation $\|x_1\|\leq K\varepsilon$ leads to $K_R\leq K,$ a contradiction to the initial supposition.
\end{proof}

\begin{remark}
 Theorem \ref{th2} is an extension of the result given in \cite{baias-popa, baias-popa3} for distinct roots of the characteristic equation.
Indeed, the particular case $p=2$ corresponds to the second order linear difference equation, i.e.,
\begin{equation}\label{eq5}
 x_{n+2}=a_1 x_{n+1}+a_2x_n,\ n\in\mathbb{N},
 \end{equation}
 and the best Ulam constant in this case is
$$K_R=\frac{1}{|r_1-r_2|}\sum\limits_{s=1}^{\infty}\left|\frac{1}{r_1^s}-\frac{1}{r_2^s}\right|,$$ for $|r_1|>1,$ $|r_2|>1$ (see \cite{baias-popa2}).

The particular case $p=3$ corresponds to the third order linear difference equation, i.e.,
\begin{equation}\label{eqp3}
x_{n+3}=a_1x_{n+2}+a_2x_{n+1}+a_3x_n,\ n\in\mathbb{N},
\end{equation}
 and the best Ulam constant in this case is
$$K_R=\frac{1}{|(r_3-r_1)(r_3-r_2)(r_2-r_1)|}\sum\limits_{s=1}^{\infty}\left|\frac{r_3-r_2}{r_1^{s}}+\frac{r_1-r_3}{r_2^{s}}+\frac{r_2-r_1}{r_3^{s}} \right|, $$ for $|r_1|>1,$ $|r_2|>1, |r_3|>1$ (see \cite{baias-popa3}).
\end{remark}

\begin{remark}
It will be interesting to obtain a closed form for the best Ulam constant of the equation $(\ref{eq1})$ for all the cases when the roots of the characteristic equation are situated outside of the unit circle.
\end{remark}





\begin{thebibliography}{10}

\bibitem{agarwal} R.P. Agarwal, B. Xu, W. Zhang, \textit{Stability of functional equations in single variable}, J. Math. Anal. Appl., \textbf{288(2)} (2003), 852--869.

\bibitem{anderson} D.R. Anderson, M. Onitsuka, {\em Best Constant for Hyers-Ulam stability of a second order h-Difference equation with constant coefficients }, Results Math, \textbf{74} 151(2019), https://doi.org/10.1007/s00025-019-1077-9.

\bibitem{baias-popa} A.R. Baias, F. Blaga, D. Popa, \textit{Best Ulam constant for a linear difference equation}, Carpathian J. Math. \textbf{35} (2019), 13--22.

\bibitem{baias-popa1} A.R. Baias, D. Popa, \textit{On Ulam stability of a linear difference equation in Banach spaces}, Bull. Malays. Math. Sci. Soc., \textbf{43} (2020), 1357--1371.

\bibitem{baias-popa2} A.R. Baias, D. Popa, \textit{On the best Ulam constant of the second order linear differential operator}, Revista de la Real Academia de Ciencias Exactas, Fisicas y Naturales. Serie A. Matematicas volume, 114, 23(2020), DOI: 10.1007/s13398-019-00776-4

\bibitem{baias-popa3} Baias, A.R., Popa, D.:\textit{On Ulam stability of a third order linear difference equation in Banach spaces,} Aequat. Math. (2020),  https://doi.org/10.1007/s00010-020-00722-5.

\bibitem{buse1} D. Barbu, C. Bu\c se, A. Tabassum, \textit{Hyers-Ulam stability and discrete dichotomy}, J. Math. Anal. Appl., \textbf{423} (2015), 1738--1752.

\bibitem{brzdeck_} N. Brillou\"{e}t-Belluot, J. Brzdek, K. Cieplinski, {\em On some recent developments in Ulam's type stability}, Abstr. Appl. Anal., 2012, 41 pp.

\bibitem{brzdek-jung} J. Brzdek, S.M. Jung, \textit{Note on stability of a linear functional equation of second order connected with the Fibonacci numbers and Lucas sequences}, J. Inequal. Appl. vol. 2010, Article ID 793947, 10 pages.

\bibitem{brzdek-popa} J. Brzdek, D. Popa, I. Ra\c sa, B. Xu, \textit{Ulam Stability of Operators}, Academic Press, 2018.

\bibitem{brzdek-popa1} J. Brzdek, D. Popa, B. Xu, \textit{Remarks on stability of linear recurrence of higher order}, Appl. Math. Lett., \textbf{23(12)} (2010), 1459--1463.

\bibitem{brzdek12} J. Brzdek, S.M. Jung, {\em A note on stability of an operator linear equation of
the second order}, Abstr. Appl. Anal. (2011), Article ID 602713, 15 pp.

\bibitem{buse2} C. Bu\c se, D. O'Regan,  O. Saierli, A. Tabassum, \textit{Hyers-Ulam stability and discrete dichotomy for difference periodic systems}, Bull. Sci. Math., \textbf{140} (2016), 908--934.

\bibitem{cull} P. Cull,  M. Flahive, R. Robson, \textit{Difference Equations. From Rabbits to Chaos}, Springer-Verlag New York, 2005.


\bibitem{dilworth} S.J. Dilworth, R. Howard, J.W. Roberts , \textit{Extremal Approximately Convex Functions and the Best Constants in a Theorem of Hyers and Ulam}, Advances in Mathematics \textbf{172(1)} (2002), 1--14.

\bibitem{elaydi} S. Elaydi, \textit{An Introduction to Difference Equations},  Springer-Verlag New York, 2005.

\bibitem{fabian} M. Fabian, P. Habala, P. H\'{a}jek, V. Montesinos Santalucia, J. Pelant, V. Zizler, \textit{Functional Analysis and Infinite-Dimensional Geometry}, Springer-Verlag New York, 2001.

\bibitem{gil} M. Gil, \textit{Difference Equations in Normed Spaces. Stability and Oscillations}, North-Holland Mathematics Studies, 2007.

\bibitem{hatori} O. Hatori, K. Kobayashi, T. Miura, H. Takagi, S.E. Takahasi, \textit{On the best constant of Hyers-Ulam stability}, Nonlinear Convex Anal., \textbf{5(3)} (2004), 387--393.

\bibitem{hyers}  D.H. Hyers, \textit{On the stability of the linear functional equation}, Proc. Natl. Acad. Sci. USA \textbf{27(4)}, 222--224, 1941.

\bibitem{hyers1}  D.H. Hyers, G. Isac, T.M. Rassias, \textit{Stability of Functional Equations in Several Variables. Progress in Nonlinear Differential Equations and Their Applications}, Birkh\"{a}user Basel, 1998.

\bibitem{jung} S.M. Jung, \textit{Functional equation $f(x) =pf(x - 1) - qf(x - 2)$ and its Hyers-Ulam stability}, J. Inequal. Appl. vol. 2009, Article ID 181678, 10 pages.

\bibitem{rassias1} S.M. Jung, M.Th. Rassias, {\em A linear functional equation of third order
associated to the Fibonacci numbers}, Art. ID 137468, 7 pp. 2014.


\bibitem{kadison} R.V. Kadison, J.R. Ringrose, \textit{Fundamentals of the Theory of Operator Algebras. Vol I: Elementary Theory}, Academic Press, 1983.

\bibitem{onitsuka1} M. Onitsuka, \textit{Influence of the Stepsize on Hyers-Ulam Stability of First-Order Homogeneous Linear Difference Equations}, Int. J. Difference Equ. \textbf{12(2)} (2017), 281--302.


\bibitem{onitsuka} M. Onitsuka, {\em Hyers Ulam stability of first-order nonhomogeneous linear difference equations with a constant stepsize}, Appl. Math. Comput., \textbf{{330(1)}} (2018), 143--151.

\bibitem{palmer} Palmer, K.J:{\em Shadowing in Dynamical Systems. Theory and Applications}, Kluwer Academic Press, (2000).

\bibitem{popa1} D. Popa, {\em Hyers-Ulam-Rassias stability of a linear recurrence}, J. Math. Anal.
Appl., \textbf{309(2)} (2005),  591--597.

\bibitem{popa} D. Popa, \textit{Hyers-Ulam stability of the linear recurrence with constant coefficients}, Adv. Differ. Equ., \textbf{2} (2005), 101--107.

\bibitem{popa3} D. Popa, I. Rasa, {\em Best constant in stability of some positive linear
operators}, Aequationes Math., \textbf{90} (2016), 719--726.

\bibitem{popa4} D. Popa, I. Rasa, {\em On the best constant in Hyers-Ulam stability of some
positive linear operators}, J.Math. Anal. Appl., \textbf{412} (2014), 103--108.


\bibitem{popa12} D. Popa, I. Rasa, \textit{Best constant in stability of some positive linear operators},  Aequat. Math. \textbf{90} (2016), 719--726.

\bibitem{Rassias} Th.M. Rassias, J. Tabor, \textit{ What is left of Hyers-Ulam stability?}, J. Natur. Geom., \textbf{1} (1992), 65--69.


\end{thebibliography}

\bibliographystyle{amsplain}

\end{document}